\documentclass[11pt,a4paper]{article}
\usepackage[T1]{fontenc}
\usepackage[utf8]{inputenc}
\usepackage{lmodern}
\usepackage[margin=25mm]{geometry}
\usepackage{amsmath,amssymb,amsthm,mathtools,bm}
\usepackage{graphicx,booktabs,enumerate,placeins,microtype}
\usepackage[numbers,sort&compress]{natbib}
\usepackage[font=small,labelfont=bf,labelsep=space]{caption}
\usepackage[hidelinks]{hyperref}
\hypersetup{pdftitle={Lifetime-unified quadratic and cubic approximations of fluorescence peak time for a time-domain diffuse-light response},pdfauthor={Junyong Eom}}
\graphicspath{{./}}
\newtheorem{theorem}{Theorem}
\newtheorem{proposition}{Proposition}

\theoremstyle{remark}

\newcommand{\erfc}{\operatorname{erfc}}
\newcommand{\erfcx}{\operatorname{erfcx}}
\newcommand{\tpeak}{t_{\mathrm{peak}}}
\newcommand{\Kop}{\mathcal{K}_{\rho}}
\newcommand{\R}{\mathbb{R}}
\newcommand{\doi}[1]{\url{https://doi.org/#1}}

\title{Lifetime-unified quadratic and cubic approximations of fluorescence peak time for a time-domain diffuse-light response}
\author{Junyong Eom\\[3pt]
\small Graduate School of Science and Technology, Gunma University\\
\small Maebashi, Gunma, Japan\\
\small Corresponding author: \href{mailto:eom@gunma-u.ac.jp}{eom@gunma-u.ac.jp}\\
\small ORCID: \href{https://orcid.org/0000-0002-2749-3322}{0000-0002-2749-3322}}
\date{}

\begin{document}
\raggedbottom
\maketitle
\begin{abstract}
We study the temporal maximum of a boundary response generated by a coupled diffusion system in a half-space with Robin boundary conditions and a point fluorophore. The fluorescence lifetime enters through a one-sided exponential time convolution. Previous peak-time approximations were derived under separate short- and long-lifetime scalings. We construct a single reduced peak-time formula for every positive ratio of lifetime to the intrinsic width of the transport peak, thereby describing the transition between these regimes without switching formulas. Starting from a known large-distance approximation of the zero-lifetime response, we normalize at its unique maximum and approximate its logarithm quadratically. The resulting scalar mode equation has a unique solution, expressed through the inverse scaled complementary error function. Its endpoint expansions recover the leading structures of the earlier approximations in the corresponding joint large-distance regimes. A cubic correction accounts for the leading local asymmetry. Under a stated bound on the normalized cubic coefficient, we prove uniqueness of its positive mode and agreement with the exact reduced-profile peak expansion through second order in the lifetime. A conditional peak-location stability estimate is combined with a decomposition into transport-profile and local truncation errors. Numerical comparisons with the full diffusion response support the accuracy of the construction. The lifetime-unified formula is distinguished from the fixed-geometry, infinite-lifetime limit, in which the transport tail yields a different peak-time growth law.
\end{abstract}


\smallskip
\noindent\textbf{Acknowledgments.} This work was supported by the Japan
Society for the Promotion of Science (JSPS), KAKENHI Grant Number JP25K17291.

\section{Introduction}
\label{sec:introduction}

The time at which a boundary trace of a linear parabolic system reaches
its maximum is a nonlinear feature of the underlying evolution. In
time-domain fluorescence diffuse optical tomography (FDOT), this \emph{peak time}
is used as reduced data for locating a fluorescent target: a short boundary
pulse diffuses to the target, fluorescence is emitted after a delay, and the
emitted light diffuses to a boundary detector. The measured temporal trace
therefore combines spatial transport and fluorescence decay
\cite{Arridge1999,ArridgeSchotland2009,Jiang2022FMT,ZhangEtAl2022}.
Unlike its amplitude, the location of a temporal maximum is unchanged by a
positive multiplicative factor. This makes peak time useful when source
strength, detector gain, or fluorophore concentration is not known precisely.
Experimental and analytical studies have related this observable to target
depth \cite{HallEtAl2004,HanHall2008,LaidevantEtAl2007}.

The mathematical question addressed here is how the fluorescence lifetime
changes the location of that maximum. Earlier analyses describe the peak
in separate short- and long-lifetime regimes. Our aim is a single reduced
peak-time formula that does not require choosing between those regimes
and that retains their leading asymptotic structures at the appropriate
endpoints. We first introduce the parabolic model and its zero-lifetime
response, so that both the earlier approximations and the transition
problem can be stated directly in terms of the governing equations.

\subsection{The parabolic model and the peak-time observable}
\label{sec:model}

Let $\Omega=\R^3_+=\{x=(x_1,x_2,x_3):x_3>0\}$, let
$x_s,x_d\in\partial\Omega$ denote the source and detector, and let
$x_c\in\Omega$ be the location of a point fluorophore. We use constant,
identical optical coefficients for excitation and emission. The unknown
fluorescence strength is represented by
$\mu_f(x)=c\delta(x-x_c)$, $c>0$. The excitation field $u_e$ and emission
field $U_m$ solve the coupled diffusion system
\cite{ChenEtAl2026IPI,ChenEtAl2026Global}
\begin{equation}
 \left\{
 \begin{aligned}
  (v^{-1}\partial_t-D\Delta+\mu_a)u_e&=0
       &&\text{in }\Omega\times(0,\infty),\\
  u_e(\cdot,0)&=0
       &&\text{in }\Omega,\\
  \partial_\nu u_e+\beta u_e&=\delta(x-x_s)\delta(t)
       &&\text{on }\partial\Omega\times(0,\infty),
 \end{aligned}
 \right.
 \label{eq:excitation}
\end{equation}
and
\begin{equation}
 \left\{
 \begin{aligned}
  (v^{-1}\partial_t-D\Delta+\mu_a)U_m
      &=\mu_f(x)(f_\ell*u_e)(x,t)
       &&\text{in }\Omega\times(0,\infty),\\
  U_m(\cdot,0)&=0
       &&\text{in }\Omega,\\
  \partial_\nu U_m+\beta U_m&=0
       &&\text{on }\partial\Omega\times(0,\infty).
 \end{aligned}
 \right.
 \label{eq:emission}
\end{equation}
Here $\nu$ is the exterior unit normal, $v>0$ is the speed of light in the
medium, $D>0$ is the scaled diffusion coefficient, $\mu_a>0$ is the
background absorption coefficient, and $\beta>0$ is the Robin boundary-loss
parameter. The convolution is in time, with
\[
 f_\ell(t)=\ell^{-1}e^{-t/\ell}\bm{1}_{t\geq0},\qquad
 (f_\ell*u_e)(x,t)=\int_0^t f_\ell(t-s)u_e(x,s)\,ds,
 \qquad \ell>0.
\]
The zero-lifetime model is defined by replacing $f_\ell*u_e$ with $u_e$.
Thus zero lifetime means instantaneous emission at the fluorophore; it does
not remove either diffusion process.

For a fixed source--detector pair, write
$U_\ell(t)=U_m(x_d,t;x_s)$. The measured feature is
\begin{equation}
 \tpeak(x_s,x_d)
 :=\underset{t>0}{\operatorname{arg\,max}}\;U_\ell(t).
 \label{eq:peak-definition}
\end{equation}
All positive time-independent multiplicative factors, including $c$, cancel
from this definition. Previous work has used peak times to construct local,
global, and direct target-location methods
\cite{EomEtAl2023JMP,EomEtAl2024DIE,ChenEtAl2023CAC,ChenEtAl2026Global,ChenEtAl2026IPI}.
This reduction differs from using temporal moments
\cite{Arridge1995Moments,DucrosEtAl2009PartI,DucrosEtAl2009PartII} or fitting
the full waveform \cite{EomKimSon2026}. The present paper concerns the
forward prediction of the peak time that such inverse methods use, rather
than a new reconstruction algorithm. The uniqueness statements proved
below concern the reduced transport response and the proposed local
models, not a general uniqueness theorem for every full diffusion trace.

\subsection{The zero-lifetime response and its transport approximation}

Let $u_m(t)=U_0(t)$ denote the zero-lifetime boundary response. The half-space
Green-function representation gives, up to a positive time-independent
factor,
\cite{LamEtAl2005,ChenEtAl2023CAC,ChenEtAl2026Global}
\begin{align}
 u_m(t)
 &=e^{-v\mu_at}\int_0^t [(t-s)s]^{-3/2}
 \exp\!\left[-\frac{|x_d-x_c|^2}{4vD(t-s)}
              -\frac{|x_s-x_c|^2}{4vDs}\right] \notag\\
 &\hspace{2.2cm}\times
 \widehat K(x_{c3};t-s)\widehat K(x_{c3};s)\,ds,
 \qquad t>0,
 \label{eq:full-zero-response}
\end{align}
Here and below, a common positive normalization of $u_m$ and $U_\ell$
is suppressed, since it does not affect peak time. The target depth is
$x_{c3}$, and
\begin{equation}
 \widehat K(z;t)
 =1-\beta\sqrt{\pi vDt}\,
 \erfcx\!\left(\frac{z+2\beta vDt}{\sqrt{4vDt}}\right).
 \label{eq:boundary-factor}
\end{equation}
We use the complementary error function and its scaled version,
\[
 \erfc(\xi)=\frac{2}{\sqrt\pi}\int_\xi^\infty e^{-s^2}\,ds,
 \qquad \erfcx(\xi)=e^{\xi^2}\erfc(\xi),\qquad \xi\in\R.
\]
The integral in \eqref{eq:full-zero-response} accounts for the two diffusion
legs and the Robin boundary factor. A nonzero fluorescence lifetime is then
introduced exactly by
\begin{equation}
 U_\ell(t)=\frac1\ell\int_0^t e^{-(t-s)/\ell}u_m(s)\,ds,
 \qquad \ell>0.
 \label{eq:lifetime-convolution}
\end{equation}
In particular, differentiation gives
$\ell U_\ell'(t)+U_\ell(t)=u_m(t)$, so a stationary point satisfies
$U_\ell(t)=u_m(t)$. This identity does not require a small- or large-lifetime
assumption. Equations~\eqref{eq:full-zero-response} and
\eqref{eq:lifetime-convolution} provide the full numerical reference in this
paper.

For analytical approximation, introduce
\begin{equation}
 \lambda^2=\frac{|x_d-x_c|^2+|x_s-x_c|^2}{2vD},
 \qquad k=v\mu_a,\qquad q=\beta vD.
 \label{eq:lambda-k-q}
\end{equation}
Under the large-distance and balanced-path assumptions established in
\cite{EomEtAl2024DIE,ChenEtAl2023CAC}, the zero-lifetime response has the
leading transport profile
\begin{equation}
 u_m^a(t)=C_0e^{-kt}t^{-3/2}e^{-\lambda^2/t}
 \left(\frac{x_{c3}}{x_{c3}+qt}\right)^2,
 \qquad t>0,
 \label{eq:asymptotic-profile}
\end{equation}
where $C_0>0$ does not affect peak time. The balance assumption limits the
asymmetry between the source--target and detector--target paths; its precise
form is recalled in Section~\ref{sec:local-geometry}. The reduced profile
has a unique positive maximizer $t_0$
\cite[Theorem~3.1]{ChenEtAl2023CAC}. It is this transport peak, rather than
a fitted Gaussian, that determines the center and scale of our construction.
Dimensional quantities are retained for interpretation. Whenever an
asymptotic order in distance is stated, it is understood in the
nondimensional variables
$\widehat\lambda=\sqrt{v/D}\,\lambda$, $\widehat t=vt/D$, and
$\widehat\ell=v\ell/D$; hats are displayed where needed.

\subsection{Earlier endpoint formulas and the transition problem}

The zero-lifetime approximation was subsequently extended to nonzero
lifetime in two related directions. The analysis in
\cite{ChenEtAl2026Global} derived a long-lifetime approximate peak equation
and used it in global reconstruction for point targets. The direct
inversion study \cite{ChenEtAl2026IPI} derived separate short- and
long-lifetime expansions for approximate peak times. To make the distinction
from the exact response explicit, denote those earlier approximations by
$t_{\mathrm{short}}^p$ and $t_{\mathrm{long}}^p$. Selected terms of
\cite[Theorems~2.3 and~2.7]{ChenEtAl2026IPI} have the form
\begin{equation}
 \begin{aligned}
 t_{\mathrm{short}}^p
 &=\frac{\lambda}{\sqrt{k}}+b_0+\ell+\cdots,
      &&\text{short-lifetime scaling},\\[1mm]
 t_{\mathrm{long}}^p
 &=\frac{\lambda}{\sqrt{\kappa}}
   +\frac{\lambda^{1/2}}{\kappa^{3/4}}
    \sqrt{\log\!\left(\frac{\ell\kappa^{3/4}}
                         {\sqrt\pi\lambda^{1/2}}\right)}+\cdots,
      &&\text{long-lifetime scaling},
 \end{aligned}
 \label{eq:previous-endpoint-formulas}
\end{equation}
where
\[
 b_0=-\frac{7}{4k}
       +\frac1k\frac{\sqrt{k}}{\sqrt{k}+\beta\sqrt{vD}},
 \qquad \kappa=k-\ell^{-1}.
\]
The short-lifetime expansion uses a small dimensionless lifetime. The
long-lifetime construction assumes
$\kappa>0$ and $\ell>\sqrt\pi\,\kappa^{-3/4}\lambda^{1/2}$.
Both lines are large-distance expansions in their respective admissible
lifetime scalings, with the hypotheses and remainders specified in the
cited results. The dots in \eqref{eq:previous-endpoint-formulas} are not a
common uniform remainder. In particular, the second line is not the
fixed-geometry limit $\ell\to\infty$.

The leading short-lifetime delay is $\ell$, whereas the long-lifetime
approximation involves a different effective rate and a square-root
logarithmic correction. These descriptions provide the two endpoints, but
do not themselves give one reduced formula through the intermediate
regime. The question is therefore not whether another endpoint term can be
calculated, but whether the lifetime dependence can be retained within a
single mode equation while the transport response is approximated.
\emph{The purpose of this paper is to construct and analyze such a
lifetime-unified peak-time approximation.}

\subsection{A lifetime-unified construction and the main results}

We retain the exponential convolution and approximate the transport profile
near its own maximum. The intrinsic temporal width and the normalized
lifetime are
\[
 \sigma=\left[-\left.\frac{d^2}{dt^2}\log u_m^a(t)
                    \right|_{t=t_0}\right]^{-1/2},
 \qquad \rho=\frac{\ell}{\sigma}.
\]
The width is well defined because the reduced transport maximum is
nondegenerate. With $y=(t-t_0)/\sigma$, the leading local logarithmic profile
is $-y^2/2$. Convolution of this Gaussian profile with the exponential
kernel yields an exponentially modified Gaussian
\cite{KalambetEtAl2011}. Its mode gives
\[
 \tpeak^{(2)}=t_0+\sigma\left[
 \rho^{-1}-\sqrt2\,\erfcx^{-1}\!\left(\sqrt{2/\pi}\,\rho\right)
 \right],\qquad \rho>0.
\]
Thus one formula covers short lifetime, the transition $\rho=O(1)$, and long
lifetime without regime-dependent switching. The small- and large-$\rho$
expansions recover the leading lifetime-dependent structures in
\eqref{eq:previous-endpoint-formulas} under the corresponding joint
large-distance limits. This is a reduction of a single convolution problem,
not an interpolation of two endpoint formulas.

The quadratic profile is symmetric. The first asymmetry of the transport
peak is described by the local logarithmic coefficient
$\gamma_3=\sigma^3(\log u_m^a)^{(3)}(t_0)/6$. Retaining that contribution yields
a cubic correction whose positive mode is determined by one scalar
nonlinear equation. Theorems~\ref{thm:quadratic} and \ref{thm:cubic}
establish the unique modes of the quadratic and cubic models; the latter
holds for $0\leq\gamma_3<1/2$. Proposition~\ref{prop:short-lifetime-consistency}
shows that the cubic approximation reproduces the exact reduced-profile
peak expansion through order $\ell^2$, whereas the quadratic model retains
the leading delay. Proposition~\ref{prop:local-orders} identifies the local
approximation orders, and Proposition~\ref{prop:peak-stability} gives a
conditional mechanism for transferring convolution residuals into
peak-location errors.

The scope of the approximation is important. The large-distance and
balanced-path assumptions on $u_m^a$ remain in force, and the local expansion
alone does not control the entire incoming convolution tail. No uniform
error theorem over arbitrary geometries and lifetimes is asserted.
In particular, at fixed geometry the exact reduced-profile peak grows
logarithmically as $\ell\to\infty$, in contrast to the square-root logarithmic
Gaussian delay. The term \emph{lifetime-unified} refers to the single reduced
construction and its joint-regime endpoint matching, not to asymptotic
accuracy in that different limit. Numerical comparisons with the full
response assess the approximation separately: across the 104 tested
configuration--lifetime pairs, the cubic correction reduces the mean
relative error from $3.280\%$ to $0.234\%$, without fitting the formulas to
reference peak times.

Section~\ref{sec:local-geometry} establishes the local geometry and
large-distance scaling of the transport peak. Section~\ref{sec:method}
derives the lifetime-unified formulas and their analytical properties.
Section~\ref{sec:limits} relates them to the earlier endpoint expansions and
distinguishes the fixed-geometry limit. Numerical validation is presented in
Section~\ref{sec:numerics}. Sections~\ref{sec:discussion} and
\ref{sec:conclusion} discuss the scope of the results and give the conclusions.

\section{Local geometry of the transport peak}
\label{sec:local-geometry}

\subsection{Large-distance profile}

We use the parameters in \eqref{eq:lambda-k-q} and the reduced profile
\eqref{eq:asymptotic-profile}, introduced in Section~\ref{sec:introduction}.
We first specify the large-distance setting and then determine the local
geometry of its unique peak.

For clarity, write $\widehat x=x/D$.  The balanced-path condition used in
\cite[Lemma~2.1]{ChenEtAl2026IPI} takes the dimensionless form
\begin{equation}
 \left|\,|\widehat x_d-\widehat x_c|^2
       -|\widehat x_s-\widehat x_c|^2\,\right|
 \leq C_b\widehat t,
 \label{eq:balanced-path}
\end{equation}
with $C_b$ independent of $\widehat\lambda$ on the time window being
approximated.  The nondimensional optical and boundary parameters,
including $D\mu_a$ and $D\beta$, are kept in compact positive sets.  A
concrete admissible family has bounded nondimensional source and detector
positions and bounded lateral target coordinates, while
$\widehat x_{c3}\to\infty$; then
$\widehat t_0\asymp\widehat\lambda\asymp\widehat x_{c3}$ and
\eqref{eq:balanced-path} holds on a peak window.  The local estimates below
concern $u_m^a$ itself.  Their transfer to the full response also requires
control of the transport approximation over the incoming convolution
interval.

Let $L(t)=\log u_m^a(t)$.  Its logarithmic slope is
\begin{equation}
 P(t):=L'(t)
 =-k-\frac{3}{2t}+\frac{\lambda^2}{t^2}
 -\frac{2q}{x_{c3}+qt}.
 \label{eq:P}
\end{equation}
Theorem~3.1 of \cite{ChenEtAl2023CAC} proves, under these assumptions, that
the equation
\begin{equation}
 P(t_0)=0
 \label{eq:t0}
\end{equation}
has a unique positive solution.  The derivative signs needed below can also be
verified directly at this root.  Set
\[
 a=\frac{qt_0}{x_{c3}+qt_0}\in(0,1),\qquad b=kt_0>0.
\]
Multiplication of \eqref{eq:t0} by $t_0$ gives
$\lambda^2/t_0=b+3/2+2a$, and substitution into the first two derivatives of
$P$ yields
\begin{equation}
 -t_0^2P'(t_0)=2b+\frac32+4a-2a^2>0,
 \qquad
 t_0^3P''(t_0)=6b+6+12a-4a^3>0.
 \label{eq:derivative-signs}
\end{equation}
Thus $P'(t_0)<0$ and $P''(t_0)>0$.  Since $P(t)\to+\infty$ as $t\downarrow0$
and $P(t)\to-k<0$ as $t\to\infty$, uniqueness of the zero implies that $t_0$
is the unique peak time of $u_m^a$, because $(u_m^a)'=Pu_m^a$ and $u_m^a>0$.

\subsection{Intrinsic width and skewness}

We define the intrinsic temporal width by
\begin{equation}
 \sigma^2=-\frac1{P'(t_0)},
 \qquad
 P'(t)=\frac{3}{2t^2}-\frac{2\lambda^2}{t^3}
 +\frac{2q^2}{(x_{c3}+qt)^2},
 \label{eq:sigma}
\end{equation}
The first identity in \eqref{eq:derivative-signs} makes this definition
well posed.  The quantity $\sigma$ has units of time and is the inverse
square root of the downward curvature of the log-profile.  In the Gaussian
approximation, moving one $\sigma$ from $t_0$ reduces the response by the
factor $e^{-1/2}$, and the full width at half maximum is approximately
$2\sqrt{2\log2}\,\sigma$.  Thus $t_0$ describes the arrival location of the
transport peak, whereas $\sigma$ describes its temporal spread.

The leading dimensionless local logarithmic cubic-asymmetry coefficient,
which we call the \emph{local logarithmic skewness}, is
\begin{equation}
 \gamma_3:=\frac{\sigma^3P''(t_0)}6,
 \qquad
 P''(t)=-\frac3{t^3}+\frac{6\lambda^2}{t^4}
 -\frac{4q^3}{(x_{c3}+qt)^3}.
 \label{eq:gamma3}
\end{equation}
It is positive for the reduced FDOT profile by
\eqref{eq:derivative-signs}; unlike statistical skewness, it is a coefficient
of the local logarithmic expansion rather than a standardized third moment.
To make the origin of the local model explicit, set
$y=(t-t_0)/\sigma$ and apply Taylor's formula to $L$ at $t_0$.
Since $L'=P$, one obtains
\begin{align*}
 L(t_0+\sigma y)-L(t_0)
 &=\sigma yP(t_0)+\frac{\sigma^2y^2}{2}P'(t_0)
   +\frac{\sigma^3y^3}{6}P''(t_0)\\
 &\quad+\frac{\sigma^4y^4}{24}P'''(t_0)+R_5(y).
\end{align*}
The first term vanishes by \eqref{eq:t0}, while
$\sigma^2P'(t_0)=-1$ by \eqref{eq:sigma}.  Thus the linear and quadratic
terms reduce exactly to $0$ and $-y^2/2$, respectively.  Defining
$\gamma_4=\sigma^4P'''(t_0)/24$ gives
\begin{equation}
 L(t_0+\sigma y)-L(t_0)
 =-\frac{y^2}{2}+\gamma_3y^3+\gamma_4y^4+R_5(y).
 \label{eq:local-log-expansion}
\end{equation}
For example, the integral remainder is
\[
 R_5(y)=\frac{\sigma^5y^5}{4!}\int_0^1(1-s)^4
 P^{(4)}(t_0+s\sigma y)\,ds,
\]
and the derivative entering $\gamma_4$ is
\[
 P'''(t)=\frac9{t^4}-\frac{24\lambda^2}{t^5}
 +\frac{12q^4}{(x_{c3}+qt)^4}.
\]
The next derivative is
\[
 P^{(4)}(t)=-\frac{36}{t^5}+\frac{120\lambda^2}{t^6}
 -\frac{48q^5}{(x_{c3}+qt)^5}.
\]
Together with the scaling estimates stated below, $t_0+s\sigma y$ remains
comparable to $t_0$, so
$P^{(4)}=O(\widehat\lambda^{-4})$ and
$\sigma^5=O(\widehat\lambda^{5/2})$ on every fixed interval
$|y|\leq Y$.  Hence
\begin{equation}
 R_5(y)=O\!\left(\widehat\lambda^{-3/2}|y|^5\right).
 \label{eq:R5-order}
\end{equation}

The four estimates below follow directly from a dominant-balance argument.
In $P(t_0)=0$, the leading terms are
$-k+\lambda^2/t_0^2$; the terms proportional to $t_0^{-1}$ and
$(x_{c3}+qt_0)^{-1}$ are one order smaller.  Hence
$t_0=\lambda/\sqrt{k}+O(1)$.  Substitution into the derivatives gives,
schematically and after the nondimensionalization stated above,
\[
 P'(t_0)=-\frac{2\lambda^2}{t_0^3}
 \bigl(1+O(\widehat\lambda^{-1})\bigr),\qquad
 P''(t_0)=\frac{6\lambda^2}{t_0^4}
 \bigl(1+O(\widehat\lambda^{-1})\bigr),
\]
and
\[
 P'''(t_0)=-\frac{24\lambda^2}{t_0^5}
 \bigl(1+O(\widehat\lambda^{-1})\bigr).
\]
Thus $-P'(t_0)$ is of first inverse order in optical distance and
$\sigma^2=-1/P'(t_0)$ grows linearly with that distance.  Multiplying the
next derivatives by $\sigma^3$ and $\sigma^4$ then produces the half- and
first-inverse orders of $\gamma_3$ and $\gamma_4$.  More precisely, for the
admissible large-distance scaling,
\[
 t_0=\frac{\lambda}{\sqrt{k}}+O(1),\qquad
 \sigma^2=\frac{\lambda}{2k^{3/2}}
 \bigl(1+O(\widehat\lambda^{-1})\bigr),
 \qquad
 \gamma_3=O(\widehat\lambda^{-1/2}),\quad
 \gamma_4=O(\widehat\lambda^{-1}).
\]
Consequently, the normalized peak is asymptotically Gaussian on every fixed
interval $|y|\leq Y$.  The cubic term is the first contribution that
distinguishes the late-time side from the early-time side.

\begin{proposition}[Local approximation orders]
\label{prop:local-orders}
Fix $Y>0$ and assume the nondimensional optical and boundary parameters remain
in compact positive sets under the large-distance scaling.  Uniformly for
$|y|\leq Y$,
\[
 \frac{u_m^a(t_0+\sigma y)}{u_m^a(t_0)}
 =e^{-y^2/2}
 \left[1+O\!\left(\widehat\lambda^{-1/2}(1+|y|^3)\right)\right]
\]
and
\[
 \frac{u_m^a(t_0+\sigma y)}{u_m^a(t_0)}
 =e^{-y^2/2}(1+\gamma_3y^3)
 +e^{-y^2/2}O\!\left(
 \widehat\lambda^{-1}(1+|y|^6)\right).
\]
\end{proposition}

\begin{proof}
The first statement follows by truncating
\eqref{eq:local-log-expansion} after its quadratic term and using
$\gamma_3=O(\widehat\lambda^{-1/2})$.  For the second, expand
$\exp(\gamma_3y^3+\gamma_4y^4+R_5)$ to first order in $\gamma_3$.
The terms $\gamma_4y^4$ and $\gamma_3^2y^6$ are both
$O(\widehat\lambda^{-1}(1+|y|^6))$ on the stated interval.  In addition,
\eqref{eq:R5-order} gives
$R_5(y)=O(\widehat\lambda^{-3/2}|y|^5)$, which is absorbed by the displayed
remainder for fixed $Y$.
\end{proof}

\section{Lifetime-unified peak-time method}
\label{sec:method}

\subsection{An exact peak-balance identity}

To distinguish the reduced transport profile from the full reference, define
its lifetime convolution by
\begin{equation}
 U_\ell^a(t):=\frac1\ell\int_0^t
 e^{-(t-s)/\ell}u_m^a(s)\,ds .
 \label{eq:asymptotic-lifetime-convolution}
\end{equation}
For every $\ell>0$, direct differentiation gives
\begin{equation}
 \ell (U_\ell^a)'(t)+U_\ell^a(t)=u_m^a(t).
 \label{eq:peak-balance-ode}
\end{equation}
Hence a stationary point satisfies $U_\ell^a(t)=u_m^a(t)$.  This identity is
exact for the lifetime convolution of the asymptotic transport profile and
contains no small- or large-lifetime assumption.

The convolution $U_\ell^a$ also has a unique strict maximum, which we denote
by $t_a$, and $t_a>t_0$.  Indeed, for $0<t\leq t_0$, strict increase of
$u_m^a$ gives $U_\ell^a(t)<u_m^a(t)$, so $(U_\ell^a)'(t)>0$.
The response is positive and tends to zero as $t\to\infty$, hence it attains
a maximum after $t_0$.  At any such stationary point,
\[
 \frac{d}{dt}\{U_\ell^a-u_m^a\}=-(u_m^a)'(t)>0.
\]
Thus every zero of $U_\ell^a-u_m^a$ after $t_0$ is a simple upward crossing,
which implies uniqueness of the maximum.

\subsection{Quadratic formula}

Substitute the quadratic approximation
$u_m^a(t_0+\sigma x)/u_m^a(t_0)\approx e^{-x^2/2}$ into the lifetime
convolution.  The physical lower limit becomes $-t_0/\sigma$ in the $x$
variable.  Because $t_0/\sigma=O(\widehat\lambda^{1/2})$, replacing it by
$-\infty$ produces an exponentially small Gaussian-tail error.  Completion
of the square then yields an exponentially modified Gaussian response, as
shown next.

Set $\rho=\ell/\sigma$, $y=(t-t_0)/\sigma$, and
$\phi(x)=e^{-x^2/2}$.  The dimensionless convolution of the quadratic input is
\begin{align*}
 Q_2(y;\rho)
 &=\frac1\rho\int_{-\infty}^{y}
 e^{-(y-x)/\rho}\phi(x)\,dx\\
 &=\frac1\rho e^{-y/\rho+1/(2\rho^2)}
 \int_{-\infty}^{y}
 \exp\!\left[-\frac{(x-\rho^{-1})^2}{2}\right]dx\\
 &=\frac1\rho\sqrt{\frac\pi2}\,
 e^{-y/\rho+1/(2\rho^2)}
 \erfc\!\left(\frac{\rho^{-1}-y}{\sqrt2}\right).
\end{align*}
The second line follows from
\[
 -\frac{y-x}{\rho}-\frac{x^2}{2}
 =-\frac{y}{\rho}+\frac1{2\rho^2}
  -\frac{(x-\rho^{-1})^2}{2}.
\]
Dividing by the local input $\phi(y)$ and using
$\erfcx(z)=e^{z^2}\erfc(z)$ gives
\[
 \frac{Q_2(y;\rho)}{\phi(y)}
 =R(y;\rho):=\frac1\rho\sqrt{\frac\pi2}\,
 \erfcx\!\left(\frac{\rho^{-1}-y}{\sqrt2}\right).
\]
The peak-balance condition $Q_2(y;\rho)=\phi(y)$ from
\eqref{eq:peak-balance-ode} is therefore
\begin{equation}
 R(y;\rho)=1.
 \label{eq:quadratic-mode-balance}
\end{equation}
The scaled complementary error function is strictly decreasing on $\R$, so
the equation has exactly one real solution.  The solution is positive because
an exponential fluorescence kernel can only delay the mode of the Gaussian
input.

\begin{theorem}[Quadratic lifetime-unified peak time]
\label{thm:quadratic}
Let $t_0$ and $\sigma$ be defined by \eqref{eq:t0} and \eqref{eq:sigma}.
For every $\ell>0$, the quadratic reduced model has the unique peak time
\begin{equation}
 \tpeak^{(2)}
 =t_0+\sigma\left[
 \frac1\rho-\sqrt2\,
 \erfcx^{-1}\!\left(\sqrt{\frac2\pi}\,\rho\right)
 \right],
 \qquad \rho=\frac\ell\sigma .
 \label{eq:quadratic-formula}
\end{equation}
For $\ell=0$, the continuous extension is
$\tpeak^{(2)}=t_0$.
\end{theorem}

\begin{proof}
Put $z=(\rho^{-1}-y)/\sqrt2$.  Equation
\eqref{eq:quadratic-mode-balance} is equivalent to
\[
 \erfcx(z)=\sqrt{\frac2\pi}\,\rho.
\]
Since $\erfcx$ is continuous and strictly decreasing from $+\infty$ to $0$,
this equation has the unique solution
$z=\erfcx^{-1}(\sqrt{2/\pi}\,\rho)$.  Solving the definition of $z$ for $y$
and then using $t=t_0+\sigma y$ yields
\eqref{eq:quadratic-formula}.  Moreover,
\[
 R(0;\rho)=Q_2(0;\rho)
 <\frac1\rho\int_{-\infty}^{0}e^{x/\rho}\,dx=1,
\]
because $e^{-x^2/2}<1$ for $x<0$.  Since $R$ is strictly increasing in $y$,
the unique root is positive.  Moreover, the reduced convolution satisfies
\[
 \rho Q_2'(y;\rho)=\phi(y)-Q_2(y;\rho)
 =\phi(y)\{1-R(y;\rho)\}.
\]
It therefore increases before the root and decreases after it, so the root is
the unique strict global maximum.  Finally, the small-$\rho$ expansion in
Section~\ref{sec:limits} gives $y=\rho+O(\rho^3)$, proving the stated
continuous extension at $\ell=0$.
\end{proof}

Formula~\eqref{eq:quadratic-formula} is universal after $t_0$ and $\sigma$
have been determined: all lifetime dependence is contained in one scalar
function of $\rho$.  Optical distance, absorption, target depth, and boundary
loss affect the transport center and width; fluorescence lifetime then adds a
normalized delay.

\subsection{Cubic correction and selection of the physical root}

The quadratic model imposes local symmetry.  Retaining the leading skewness
term from Proposition~\ref{prop:local-orders} gives, for a generic coefficient
$\gamma$,
\[
 g_\gamma(x):=e^{-x^2/2}(1+\gamma x^3).
\]
This is a local asymptotic profile, not a globally positive probability
density.  Its value is that all convolution integrals remain explicit.
Define
\begin{align}
 S_3(y;\rho)
 &:=-\frac{y^2}{\rho}-\frac{y}{\rho^2}
 +\frac{3R(y;\rho)-2}{\rho}
 +\frac{R(y;\rho)-1}{\rho^3},
 \label{eq:S3}\\
 F(y,\gamma;\rho)
 &:=R(y;\rho)-1+\gamma\{S_3(y;\rho)-y^3\}.
 \label{eq:cubic-F}
\end{align}

\begin{theorem}[Cubic lifetime-unified peak time]
\label{thm:cubic}
For every $\rho>0$ and $0\leq\gamma_3<1/2$, the equation
\begin{equation}
 F(y_3,\gamma_3;\rho)=0
 \label{eq:cubic-equation}
\end{equation}
has exactly one positive root.  This root is simple, is the strict maximum of
the cubically reduced convolution on $[0,\infty)$, depends continuously on
$\gamma_3$, and satisfies $y_3\to y_2$ as $\gamma_3\to0$, where
\[
 y_2=\rho^{-1}-\sqrt2\,
 \erfcx^{-1}\!\left(\sqrt{2/\pi}\,\rho\right).
\]
The physical cubic peak approximation is
\begin{equation}
 \tpeak^{(3)}=t_0+\sigma y_3.
 \label{eq:cubic-formula}
\end{equation}
For $\ell=0$, its continuous extension is $\tpeak^{(3)}=t_0$.
\end{theorem}

\begin{proof}
Write $\phi(x)=e^{-x^2/2}$ and, for $n\geq0$, define
\[
 J_n(y;\rho)=\frac1\rho\int_{-\infty}^{y}
 e^{-(y-x)/\rho}x^n\phi(x)\,dx.
\]
Thus $J_0=Q_2$ and the convolution of $g_\gamma$ is
$Q_\gamma:=J_0+\gamma J_3$.

We first reduce $J_3$ to $J_0$.  Since
\[
 x^n\phi(x)=(n-1)x^{n-2}\phi(x)
 -\frac{d}{dx}\{x^{n-1}\phi(x)\},
\]
integration by parts, with the boundary contribution at $-\infty$ equal to
zero, gives for $n\geq1$
\[
 J_1=\frac{J_0-e^{-y^2/2}}\rho,
 \qquad
 J_n=-\frac{y^{n-1}}\rho e^{-y^2/2}
 +(n-1)J_{n-2}+\frac1\rho J_{n-1}.
\]
Because $J_0/\phi(y)=R(y;\rho)$, the first two applications of the recurrence
give
\[
 \frac{J_1}{\phi(y)}=\frac{R-1}{\rho},\qquad
 \frac{J_2}{\phi(y)}=-\frac y\rho+R+\frac{R-1}{\rho^2}.
\]
A third application yields
\begin{align*}
 \frac{J_3}{\phi(y)}
 &=-\frac{y^2}{\rho}+2\frac{J_1}{\phi(y)}
   +\frac1\rho\frac{J_2}{\phi(y)}\\
 &=-\frac{y^2}{\rho}-\frac y{\rho^2}
   +\frac{3R-2}{\rho}+\frac{R-1}{\rho^3}
 =S_3(y;\rho).
\end{align*}
The peak-balance equation for $g_\gamma=\phi(1+\gamma y^3)$ is
\[
 J_0+\gamma J_3=\phi(y)(1+\gamma y^3).
\]
After division by $\phi(y)$, this becomes precisely
$F(y,\gamma;\rho)=0$.

It remains to identify the physical root globally on the positive half-line.
Put $\Psi_\gamma=Q_\gamma-g_\gamma=\phi F$.  For $x<0$ and $\gamma\geq0$,
$g_\gamma(x)<1$, and hence
\[
 Q_\gamma(0)=\frac1\rho\int_{-\infty}^{0}
 e^{x/\rho}g_\gamma(x)\,dx<1=g_\gamma(0).
\]
On the other hand,
\[
 \lim_{y\to\infty}e^{y/\rho}Q_\gamma(y)
 =\frac{\sqrt{2\pi}}\rho e^{1/(2\rho^2)}
 \left[1+\gamma\left(\rho^{-3}+3\rho^{-1}\right)\right]>0.
\]
Because $g_\gamma$ has Gaussian decay, $\Psi_\gamma(y)>0$ for all sufficiently
large $y$.  A positive zero therefore exists.

The convolution satisfies $\rho Q_\gamma'+Q_\gamma=g_\gamma$.  At any positive
zero of $\Psi_\gamma$, this gives $Q_\gamma'=0$ and therefore
\[
 \Psi_\gamma'=-g_\gamma',\qquad
 g_\gamma'(y)=e^{-y^2/2}y\{-1+\gamma(3y-y^3)\}.
\]
For $y>0$, $3y-y^3\leq2$.  Thus $g_\gamma'(y)<0$ whenever
$0\leq\gamma<1/2$, so every positive zero of $\Psi_\gamma$ is an upward
crossing.  Two such zeros are impossible: after the first upward crossing the
function would have to cross downward before reaching the negative side of a
second upward crossing.  The positive zero is therefore unique.  Moreover,
$Q_\gamma'=-\Psi_\gamma/\rho$, so $Q_\gamma$ increases before this zero and
decreases after it on $[0,\infty)$.

At the root, $\partial_y(\phi F)=\Psi_\gamma'=-g_\gamma'>0$; hence the
implicit-function theorem gives continuous dependence on $\gamma$.  At
$\gamma=0$ the unique root is $y_2$ by Theorem~\ref{thm:quadratic}, which also
proves $y_3\to y_2$ as $\gamma\to0$.  Finally, the one-sided exponential
kernel is an approximate identity as $\rho\downarrow0$.  Since $g_\gamma$ has
its unique maximum at zero for $0\leq\gamma<1/2$, the positive mode satisfies
$y_3\to0$, giving the stated extension at $\ell=0$.
\end{proof}

The cubic equation may have a remote negative root caused by extending a local
polynomial model to the entire line.  The theorem removes any ambiguity on the
positive half-line: that negative root is not the fluorescence peak.  For each
fixed $\rho>0$, a first-order expansion as $\gamma_3\to0$ is
\begin{equation}
 y_3=y_2-\frac{\gamma_3}{y_2}
 \left(\frac{1-y_2^2}{\rho}-\frac{y_2}{\rho^2}-y_2^3\right)
 +O(\gamma_3^2).
 \label{eq:explicit-cubic}
\end{equation}
Indeed, differentiating $F(y_3(\gamma),\gamma;\rho)=0$ at $\gamma=0$
gives
$y_3'(0)=-[S_3(y_2;\rho)-y_2^3]/y_2$; inserting
$R(y_2;\rho)=1$ produces \eqref{eq:explicit-cubic}.  Independently, for fixed
$0\leq\gamma_3<1/2$, expanding
$Q_{\gamma_3}=(1+\rho\,\partial_y)^{-1}g_{\gamma_3}$ and matching powers in
the peak-balance equation gives, as $\rho\to0$,
\begin{equation}
\begin{aligned}
 y_3={}&\rho+3\gamma_3\rho^2-\rho^3
 +9\gamma_3(3\gamma_3^2-2)\rho^4\\
 &+(42\gamma_3^2+4)\rho^5+O(\rho^6).
\end{aligned}
 \label{eq:cubic-small-rho}
\end{equation}
Direct evaluation of the closed balance function
\eqref{eq:cubic-F} suffers subtractive cancellation as $\rho\to0$.  The
included implementation therefore uses \eqref{eq:cubic-small-rho} below a
prescribed small-$\rho$ threshold and the positive-root bracketing solve
otherwise.

The quadratic and cubic expansions can also be compared directly with the
exact convolution of the reduced transport profile at short lifetime.

\begin{proposition}[Short-lifetime consistency]
\label{prop:short-lifetime-consistency}
Fix the geometry and optical parameters, and suppose $\gamma_3<1/2$.
The unique peak $t_a$ of $U_\ell^a$ satisfies
\begin{equation}
 t_a=t_0+\ell-\frac{P''(t_0)}{2P'(t_0)}\ell^2+O(\ell^3)
 =t_0+\ell+\frac{3\gamma_3}{\sigma}\ell^2+O(\ell^3),
 \qquad \ell\downarrow0.
 \label{eq:short-lifetime-consistency}
\end{equation}
Consequently, the cubic approximation reproduces the expansion through
order $\ell^2$, whereas the quadratic approximation retains only its
leading lifetime delay.
\end{proposition}

\begin{proof}
Put $u=u_m^a$.  Repeated integration by parts in the exponential convolution
gives, uniformly near $t_0$,
\[
 (U_\ell^a)'(t)=u'(t)-\ell u''(t)+\ell^2u'''(t)+O(\ell^3).
\]
The derivatives of $u$ vanish at $t=0$, so there are no initial boundary
terms.  Since $u''(t_0)=P'(t_0)u(t_0)<0$, the stationary point lies at
$t_a=t_0+\ell+c_2\ell^2+O(\ell^3)$.  Substitution into the preceding
identity gives $u''(t_0)c_2+\tfrac12u'''(t_0)=0$.  At $t_0$,
$u'''(t_0)=P''(t_0)u(t_0)$, and therefore
$c_2=-P''(t_0)/(2P'(t_0))=3\gamma_3/\sigma$.
Equation~\eqref{eq:cubic-small-rho} has the same coefficient after
$\rho=\ell/\sigma$ is substituted.  The Gaussian model has no $\ell^2$
term, as also seen in \eqref{eq:small-lifetime}.
\end{proof}

\subsection{Peak stability and error decomposition}

For a normalized input profile $g$, introduce
\[
 (\Kop g)(y)=\frac1\rho\int_{-\infty}^{y}
 e^{-(y-x)/\rho}g(x)\,dx,
 \qquad
 \Phi_g(y)=(\Kop g)(y)-g(y).
\]
The reduced peak is a zero of $\Phi_g$.

\begin{proposition}[Conditional peak-location stability]
\label{prop:peak-stability}
Let $g$ and $\widetilde g$ have simple peak-balance roots $y_g$ and
$y_{\widetilde g}$ in an interval $I$.  Suppose
$|\Phi_g'(y)|\geq m>0$ between the two roots and
\[
 \varepsilon_I:=\sup_{y\in I}
 |\Phi_g(y)-\Phi_{\widetilde g}(y)|<\infty.
\]
Then
\begin{equation}
 |y_g-y_{\widetilde g}|\leq\frac{\varepsilon_I}{m}.
 \label{eq:root-stability}
\end{equation}
If $\|g-\widetilde g\|_{L^\infty(( -\infty,\sup I])}\leq\varepsilon$, then
$\varepsilon_I\leq2\varepsilon$ because $\Kop$ is a positive contraction in
$L^\infty$.
\end{proposition}

\begin{proof}
At $y_{\widetilde g}$,
$|\Phi_g(y_{\widetilde g})|\leq\varepsilon_I$.  Since
$\Phi_g(y_g)=0$, the mean-value theorem and the derivative lower bound give
\eqref{eq:root-stability}.  The contraction estimate follows directly from
the normalized exponential kernel.
\end{proof}

To state precisely what the local expansion does and does not imply, set
\[
 g_a(x)=\frac{u_m^a(t_0+\sigma x)}{u_m^a(t_0)}
 \bm{1}_{x>-t_0/\sigma},\qquad
 g_2(x)=e^{-x^2/2},\qquad
 g_3(x)=e^{-x^2/2}(1+\gamma_3x^3).
\]
Let $t_a=t_0+\sigma y_a$ be the peak of $U_\ell^a$.  Fix $Y>0$.  For
$j=2,3$, let $I\subset[-Y,Y]$ contain $y_a$ and $y_j$, and assume the root is
nondegenerate there with $|\Phi_{g_a}'|\geq m_j>0$.  If, in addition, the
full convolution residuals satisfy
\begin{equation}
 \sup_{y\in I}|\Phi_{g_a}(y)-\Phi_{g_2}(y)|
 \leq C_Y\widehat\lambda^{-1/2}(1+Y^3),
 \quad
 \sup_{y\in I}|\Phi_{g_a}(y)-\Phi_{g_3}(y)|
 \leq C_Y\widehat\lambda^{-1}(1+Y^6),
 \label{eq:convolution-residual-assumption}
\end{equation}
where $C_Y$ and $m_j$ may depend on $Y$, the fixed lifetime ratio $\rho$, and
the admissible compact parameter set, but not on $\widehat\lambda$.  Then
Proposition~\ref{prop:peak-stability} gives the conditional orders
\begin{equation}
 \frac{|t_a-\tpeak^{(2)}|}{\sigma}
 \leq \frac{C_Y}{m_2}\widehat\lambda^{-1/2}(1+Y^3),
 \qquad
 \frac{|t_a-\tpeak^{(3)}|}{\sigma}
 \leq \frac{C_Y}{m_3}\widehat\lambda^{-1}(1+Y^6).
 \label{eq:formal-error-orders}
\end{equation}
For the Gaussian model, for example,
$\Phi_{g_2}'(y_2)=y_2e^{-y_2^2/2}>0$.  Proposition~\ref{prop:local-orders}
supplies the displayed approximation rates on $|x|\leq Y$, but it does not by
itself prove \eqref{eq:convolution-residual-assumption}: $\Phi_g(y)$ depends on
the entire incoming tail $x\leq y$.  A global theorem must additionally control
the early-time tail, the replacement of $-t_0/\sigma$ by $-\infty$, and root
nondegeneracy.  Equation~\eqref{eq:formal-error-orders} is therefore explicitly
conditional, rather than a claimed consequence of a merely local Taylor
expansion.
Simplicity of $t_a$ for each fixed configuration, established in
Section~\ref{sec:method}, does not supply a lower bound $m_j$ independent of
$\widehat\lambda$ throughout the comparison interval.  In particular,
neither the residual estimate nor such a derivative bound is proved here
uniformly over a compact range of positive lifetime ratios.  The full-model
accuracy reported below is a separate numerical result.

The total discrepancy from the full model separates as
\begin{equation}
 \tpeak^{\mathrm{full}}-\tpeak^{(j)}
 =\underbrace{(\tpeak^{\mathrm{full}}-t_a)}_{\text{transport-profile error}}
 +\underbrace{(t_a-\tpeak^{(j)})}_{\text{local shape error}},
 \qquad j=2,3.
 \label{eq:error-decomposition}
\end{equation}
This decomposition is used in Section~\ref{sec:numerics} to identify why the
quadratic error grows with lifetime.

\subsection{Evaluation algorithm}

For a given source--detector pair and lifetime, the reduced calculation is:

\begin{enumerate}[1.]
 \item compute $\lambda$, $k$, and $q$ from \eqref{eq:lambda-k-q};
 \item solve $P(t_0)=0$ and evaluate $\sigma=(-P'(t_0))^{-1/2}$;
 \item set $\rho=\ell/\sigma$.  For sufficiently small $\rho$, evaluate $y_2$
 from \eqref{eq:small-rho}; otherwise use
 \eqref{eq:quadratic-formula};
 \item for the cubic result, compute $\gamma_3$.  For sufficiently small
 $\rho$, evaluate $y_3$ from \eqref{eq:cubic-small-rho}; otherwise solve
 $F(y,\gamma_3;\rho)=0$ with a bracketing root finder on the positive
 half-line.  Start with $[0,\max\{1,y_2+1\}]$ and repeatedly double the upper
 endpoint until $F>0$ there.
\end{enumerate}

Only scalar root solves are needed; no time stepping and no temporal
convolution are required.  Theorem~\ref{thm:cubic} guarantees termination of
the positive-root bracketing procedure when $0\leq\gamma_3<1/2$.  The supplied
code uses the small-$\rho$ series for $0<\rho\leq0.02$ and the closed balance
functions otherwise; this avoids cancellation in both the quadratic and cubic
evaluations.  The threshold is a floating-point implementation choice, not a
physical regime boundary.

\section{Connection with short- and long-lifetime asymptotics}
\label{sec:limits}

The quadratic mode equation provides a single reduced description across
the lifetime transition.  Its endpoint expansions are compared below with
the regime-specific approximations in
\eqref{eq:previous-endpoint-formulas} under the corresponding joint
large-distance scalings.
Figure~\ref{fig:master-curve} shows the normalized delay
$y_2=(\tpeak^{(2)}-t_0)/\sigma$ as a function of $\rho$.

\begin{figure}[!htbp]
 \centering
 \includegraphics[width=0.82\textwidth]{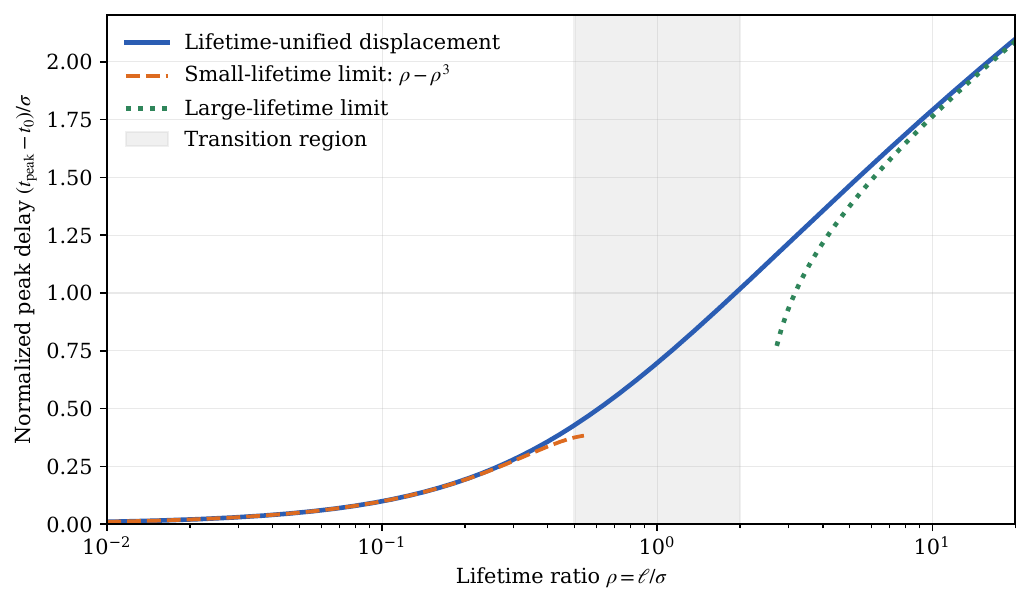}
 \caption{Dimensionless peak displacement of the quadratic reduced model.
 The solid curve shows the lifetime-unified displacement.  The dashed and
 dotted curves show the short-lifetime approximation $\rho-\rho^3$ and the
 long-lifetime approximation
 $\rho^{-1}+\sqrt{2\log(\rho/\sqrt{2\pi})}$, respectively, plotted only for
 $\rho\leq0.55$ and $\rho\geq2.7$.  The shaded band illustrates the transition
 region; its boundaries and the plotting cutoffs are illustrative, not
 rigorously established validity boundaries}
 \label{fig:master-curve}
\end{figure}

For $z\to+\infty$, the standard expansion of $\erfcx$ is
\cite{OlverEtAl2010}
\[
 \erfcx(z)=\frac1{\sqrt\pi z}
 \left(1-\frac1{2z^2}+\frac3{4z^4}
 -\frac{15}{8z^6}+\frac{105}{16z^8}+O(z^{-10})\right).
\]
Substituting this expansion into the quadratic mode equation and matching
powers of $\rho$ gives
\begin{equation}
 y_2=\rho-\rho^3+4\rho^5-27\rho^7+O(\rho^9),
 \qquad \rho\to0,
 \label{eq:small-rho}
\end{equation}
and hence
\begin{equation}
 \tpeak^{(2)}
 =t_0+\ell-\frac{\ell^3}{\sigma^2}
 +4\frac{\ell^5}{\sigma^4}
 -27\frac{\ell^7}{\sigma^6}
 +O\!\left(\frac{\ell^9}{\sigma^8}\right).
 \label{eq:small-lifetime}
\end{equation}
The leading lifetime delay is $\ell$, recovering the leading lifetime-dependent term in the first line of
\eqref{eq:previous-endpoint-formulas} and the derivative-expansion result of
\cite{ChenEtAl2026IPI}, while the sharper intrinsic regime criterion is
$\ell\ll\sigma$.

\begin{samepage}
For $\rho>\sqrt{2\pi}$, put
$a=(y_2-\rho^{-1})/\sqrt2>0$ and
$L_\rho=\log(\rho/\sqrt{2\pi})>0$. The reflection identity
$\erfcx(-a)=2e^{a^2}-\erfcx(a)$ and the quadratic mode equation give
\[
 0<a^2-L_\rho
 =\log\!\left(1+\frac{\erfcx(a)}{\sqrt{2/\pi}\,\rho}\right)
 <\frac{1}{\sqrt2\,\rho a},
\]
where $0<\erfcx(a)<1/(\sqrt\pi\,a)$ follows directly from its integral
representation. In particular, $a>\sqrt{L_\rho}$, and therefore
\[
 0<y_2-\rho^{-1}-\sqrt{2L_\rho}
 =\frac{\sqrt2\,(a^2-L_\rho)}{a+\sqrt{L_\rho}}
 <\frac{1}{2\rho L_\rho}.
\]
Consequently,
\begin{equation}
 y_2=\frac1\rho+
 \sqrt{2\log\!\left(\frac{\rho}{\sqrt{2\pi}}\right)}
 +O\!\left(\frac{1}{\rho\log\rho}\right),
 \qquad \rho\to\infty.
 \label{eq:large-rho}
\end{equation}
\end{samepage}
The remainder is thus $o(\rho^{-1})$, so the displayed $\rho^{-1}$ term is
asymptotically resolved. Returning to dimensional time gives
\begin{equation}
 \tpeak^{(2)}
 =t_0+\frac{\sigma^2}{\ell}
 +\sqrt2\,\sigma
 \sqrt{\log\!\left(\frac{\ell}{\sqrt{2\pi}\sigma}\right)}
 +O\!\left(\frac{\sigma^2}{\ell\log(\ell/\sigma)}\right),
 \qquad \frac{\ell}{\sigma}\to\infty.
 \label{eq:large-lifetime}
\end{equation}
The comparison with the earlier long-lifetime formula is a joint
large-distance and lifetime comparison.  With $\kappa=k-\ell^{-1}$, the
earlier approximation assumes
\begin{equation}
 \kappa>0,\qquad
 \ell>\sqrt\pi\,\kappa^{-3/4}\lambda^{1/2};
 \label{eq:previous-long-lifetime-conditions}
\end{equation}
see \cite[Definition~2.6 and Theorem~2.7]{ChenEtAl2026IPI}.  The square-root
logarithmic correction in \cite[Theorem~2.7]{ChenEtAl2026IPI} can be written
explicitly as
\[
 \kappa^{-3/4}\alpha_\lambda\lambda^{1/2}
 =\frac{\lambda^{1/2}}{\kappa^{3/4}}
 \sqrt{\log\!\left(
 \frac{\ell\kappa^{3/4}}{\sqrt{\pi}\lambda^{1/2}}
 \right)}.
\]
Thus the coefficient $\alpha_\lambda$ already contains the logarithmic
dependence displayed in \eqref{eq:previous-endpoint-formulas}.
The large-distance relations
$t_0=\lambda/\sqrt{k}+O(1)$ and
$\sigma^2=\lambda(2k^{3/2})^{-1}
(1+O(\widehat\lambda^{-1}))$ supply the first two terms in
\[
 \frac{\lambda}{\sqrt{k-\ell^{-1}}}
 =\frac{\lambda}{\sqrt{k}}
  +\frac{\lambda}{2k^{3/2}\ell}
  +O\!\left(\frac{\lambda}{k^{5/2}\ell^2}\right),
 \qquad k\ell\to\infty.
\]
Thus only these displayed distance-dependent terms, rather than the full
quantities without remainder control, are being identified.  Likewise,
\[
 \sqrt2\,\sigma
 \sqrt{\log\!\left(\frac{\ell}{\sqrt{2\pi}\sigma}\right)}
 \sim \frac{\lambda^{1/2}}{k^{3/4}}
 \sqrt{\log\!\left(
 \frac{\ell k^{3/4}}{\sqrt{\pi}\lambda^{1/2}}
 \right)}.
\]
In the joint regime \eqref{eq:joint-long-lifetime}, $k\ell\to\infty$ and
$\kappa/k=1-(k\ell)^{-1}\to1$.  Consequently, the displayed correction
from \cite[Theorem~2.7]{ChenEtAl2026IPI} has the same leading form; see
also \cite{ChenEtAl2026Global}.  This comparison identifies the leading
structures, without asserting equality of the full expansions or their
remainder orders.  The contribution of the reduced mode equation is thus
to retain the previously obtained endpoint structures within one
description organized by $\rho$, rather than to introduce a new
logarithmic term.

Because $y_2=O(\sqrt{\log\rho})$ at large $\rho$, a convenient quadratic
validity indicator is
\begin{equation}
 \widehat\lambda^{-1/2}(1+|y_2|^3)\ll1,
 \quad\text{or approximately}\quad
 \frac{(\log\rho)^{3/2}}{\widehat\lambda^{1/2}}\ll1.
 \label{eq:validity-indicator}
\end{equation}
The cubic approximation removes this leading skewness contribution, but its
next-order validity still deteriorates for a shallow target or an extremely
large lifetime ratio.
At the level of the local profile expansion, a compatible joint limit for
the long-lifetime comparison is
\begin{equation}
 \widehat\lambda\to\infty,\qquad \rho\to\infty,\qquad
 \frac{(\log\rho)^3}{\widehat\lambda}\to0.
 \label{eq:joint-long-lifetime}
\end{equation}
This condition keeps the quadratic validity indicator small.  It is not,
by itself, a bound on the full convolution residual in
\eqref{eq:convolution-residual-assumption} or on the full-response peak error.

The fixed-geometry limit is different.  If the geometry and optical
parameters are held fixed, then $q>0$ and
$u_m^a(t)\sim A_\infty t^{-7/2}e^{-kt}$ as $t\to\infty$, with
$A_\infty=C_0x_{c3}^2/q^2>0$.  Put
$M_a=\int_0^\infty u_m^a(s)\,ds$.  Comparing $U_\ell^a-u_m^a$ first at any
fixed positive time and then at $2k^{-1}\log(k\ell)$ shows that
$t_a\to\infty$ and $t_a=O(\log(k\ell))$.  Hence $t_a/\ell\to0$, and peak
balance gives $u_m^a(t_a)\sim M_a/\ell$.  Taking logarithms yields
\begin{equation}
 t_a\sim k^{-1}\log(k\ell),\qquad
 \ell\to\infty\quad\text{with geometry and optical parameters fixed.}
 \label{eq:fixed-geometry-long-lifetime}
\end{equation}
This logarithmic growth differs from the square-root logarithmic Gaussian
delay in \eqref{eq:large-lifetime}.  Thus lifetime unification refers to a
single reduced construction and its joint-regime endpoint matching; it is
not an assertion of asymptotic accuracy at arbitrarily large lifetime for a
fixed geometry.

\FloatBarrier
\section{Numerical methods and results}
\label{sec:numerics}

\subsection{Reference computation and reproducibility}

The baseline parameters, representative of the values used in
\cite{ChenEtAl2026IPI}, are
\[
 v=0.219\ \mathrm{mm/ps},\qquad
 D=\frac13\ \mathrm{mm},\qquad
 \mu_a=0.1\ \mathrm{mm}^{-1},\qquad
 \beta=0.5493\ \mathrm{mm}^{-1},
\]
with
\[
 x_s=(6,10,0),\qquad x_d=(14,10,0),\qquad
 x_c=(10,10,20)\ \mathrm{mm}.
\]
They give
$t_0=457.1005\,\mathrm{ps}$,
$\sigma=95.3608\,\mathrm{ps}$, and
$\gamma_3=0.1080$.

The full reference first evaluates \eqref{eq:full-zero-response}, then applies
the exponential convolution \eqref{eq:lifetime-convolution}.  In the baseline
lifetime sweep, the time step is $0.1\,\mathrm{ps}$ on
$[0,2500]$ ps, the integral after the substitution $s=tx$ is evaluated with a
240-point Gauss--Legendre rule, and each discrete maximum
is refined by three-point parabolic interpolation.  In the one-at-a-time
centered-geometry parameter sweep, the time step is $0.25\,\mathrm{ps}$ and the
quadrature order is 180.  Repeating all 104 reference-peak calculations with
$h=0.125\,\mathrm{ps}$ and quadrature order 240 changes any peak by at most
$2.13\times10^{-4}\,\mathrm{ps}$; the maximum relative change is
$5.56\times10^{-5}\%$.

For reproducibility, the exponential convolution on a uniform grid
$t_n=nh$ is updated exactly for the piecewise-linear interpolant of its input.
With $r=h/\ell$, $a=e^{-r}$,
\[
 A=\frac{1-a}{r}-a,\qquad B=1-\frac{1-a}{r},
\]
the recurrence is
\begin{equation}
 U_n=aU_{n-1}+A u_{n-1}+B u_n,
 \qquad U_0=0.
 \label{eq:convolution-recurrence}
\end{equation}
The implementation evaluates $1-a$ with \texttt{expm1}; for $\ell=0$ it sets
$U_n=u_n$.  Reported relative errors use
$100|t_{\rm approx}-t_{\rm full}|/t_{\rm full}$, and the 95th percentile uses
linear interpolation between ordered samples.  The
\path{anc/reproduce_results.py} script supplied with this preprint generates
\texttt{baseline\_results.csv}, \texttt{parameter\_sweep\_results.csv}, and
\texttt{aggregate\_summary.csv} from these specifications; it also writes
\texttt{refinement\_check.csv} for the reported grid-refinement test,
\texttt{signed\_error\_diagnostics.csv} for the signed decomposition in
\eqref{eq:error-decomposition}, and regenerates all five PDF figures. The artwork is generated with Matplotlib.

\FloatBarrier

\subsection{Baseline lifetime sweep}

Figure~\ref{fig:quadratic-full} compares the quadratic formula with the full
response for $\ell=0,100,\ldots,1000\,\mathrm{ps}$.  The formula follows the
monotone lifetime trend but increasingly underestimates the peak after the
transition region.

\begin{figure}[!htbp]
 \centering
 \includegraphics[width=0.88\textwidth]{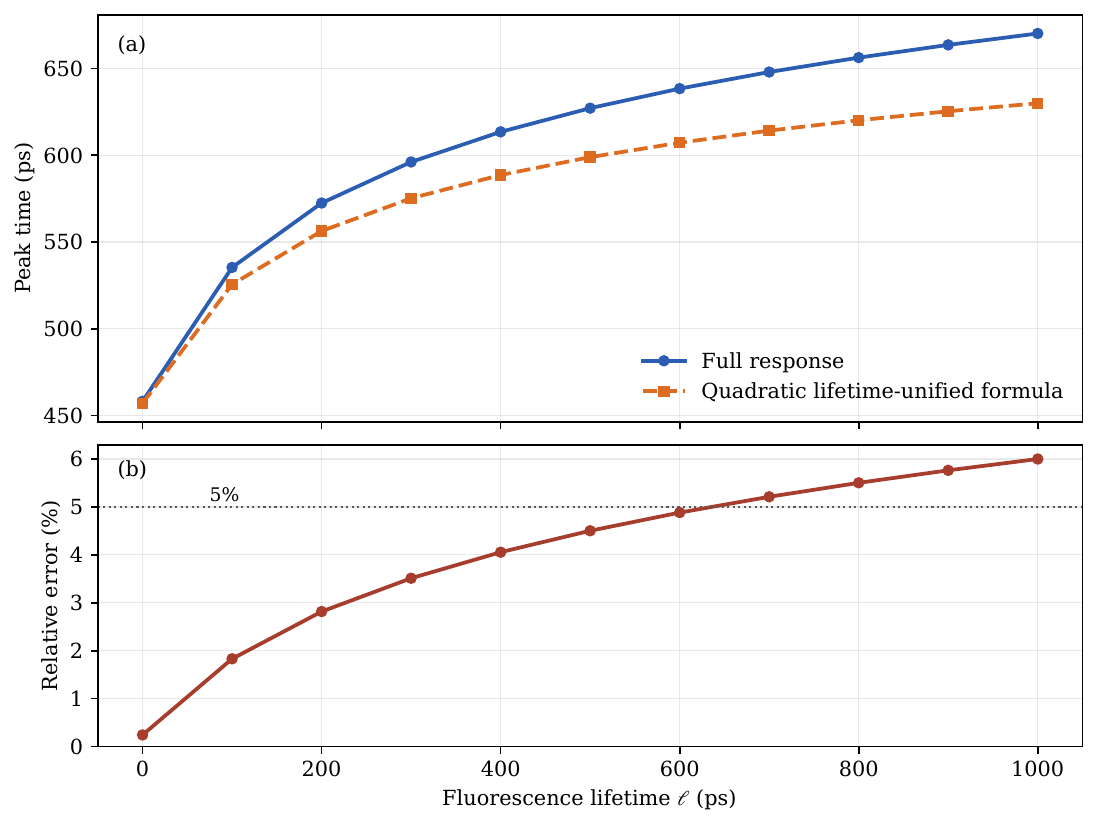}
 \caption{Quadratic lifetime-unified approximation and the full response in
 the baseline configuration. Panel (a) shows the peak times; panel (b) shows the absolute relative error}
 \label{fig:quadratic-full}
\end{figure}

Table~\ref{tab:baseline-errors} summarizes the baseline errors.
The maximum error is $5.999\%$ and the mean error is $4.028\%$.  The cubic
term corrects the systematic late-time bias (Figure~\ref{fig:quadratic-cubic})
and reduces the maximum and mean errors to $0.468\%$ and $0.213\%$,
respectively.

\begin{figure}[!htbp]
 \centering
 \includegraphics[width=0.82\textwidth]{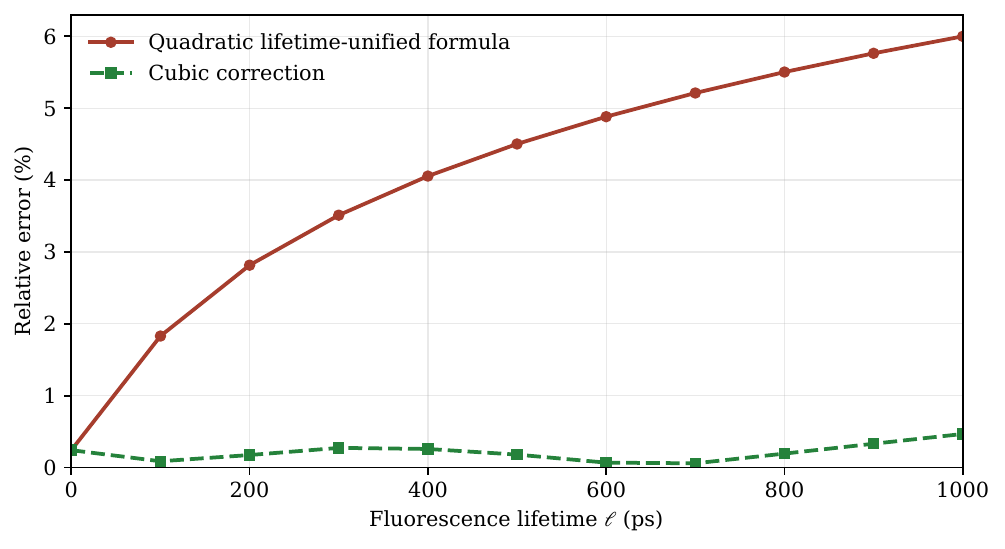}
 \caption{Relative peak-time errors of the quadratic and cubic formulas in
 the baseline lifetime sweep}
 \label{fig:quadratic-cubic}
\end{figure}

\begin{table}[!htbp]
 \caption{Baseline errors over
 $\ell=0,100,\ldots,1000\,\mathrm{ps}$.\label{tab:baseline-errors}}
 \centering
 \small
 \begin{tabular*}{0.88\textwidth}{@{\extracolsep\fill}lcc@{}}
  \toprule
  Approximation & Mean error (\%) & Maximum error (\%)\\
  \midrule
  Quadratic formula & 4.028 & 5.999\\
  Cubic formula & 0.213 & 0.468\\
  Direct convolution of $u_m^a$ & 0.242 & 0.251\\
  \bottomrule
 \end{tabular*}
\end{table}

\FloatBarrier
\subsection{Where the error comes from}

Figure~\ref{fig:error-decomposition} evaluates the two terms of
\eqref{eq:error-decomposition}.  In the baseline lifetime sweep, the relative
peak-time error introduced by replacing the full zero-lifetime response
with $u_m^a$ remains close to $0.25\%$.  Its weak lifetime dependence contrasts
with the growing quadratic local-shape error, which increases steadily with
$\rho$ and reaches more than $5\%$.
The cubic local-shape difference is nonmonotone and remains below $0.52\%$;
after its intermediate-ratio maximum, it decreases before rising modestly at
the final two ratios.  The curves are componentwise absolute differences,
so they are diagnostic rather than an additive partition of absolute errors.

\begin{figure}[!htbp]
 \centering
 \includegraphics[width=0.84\textwidth]{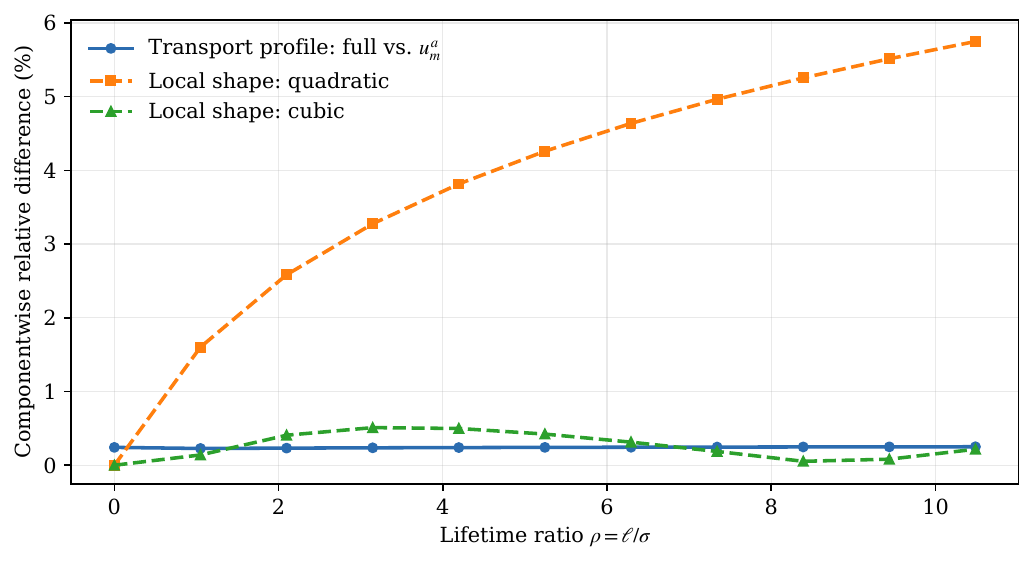}
 \caption{Componentwise error diagnostic for the baseline configuration.  The
 transport curve compares the full response with the direct convolution of
 $u_m^a$; the local-shape curves compare that convolution with the quadratic
 and cubic formulas}
 \label{fig:error-decomposition}
\end{figure}

This result supports the mathematical interpretation: the growth of the
quadratic error is primarily caused by applying a symmetric Gaussian farther
into the skewed right side of $u_m^a$, not by using the wrong lifetime scale.

\FloatBarrier
\subsection{One-at-a-time optical and centered-geometry sweep}

To test whether the baseline result is isolated, we varied one parameter at a
time around the baseline:
\[
 x_{c3}\in\{10,15,20,30,40\}\ \mathrm{mm},\quad
 \mu_a\in\{0.05,0.10,0.15,0.20\}\ \mathrm{mm}^{-1},
\]
\[
 |x_d-x_s|\in\{4,8,12,16\}\ \mathrm{mm},\quad
 \beta\in\{0.25,0.5493,1.0\}\ \mathrm{mm}^{-1}.
\]
After duplicate baseline cases are removed, this gives 13 configurations.
For a separation $d$, the source and detector were placed at
$x_s=(10-d/2,10,0)$ and $x_d=(10+d/2,10,0)$, while
$x_c=(10,10,x_{c3})$.  Thus the target remains directly beneath the
source--detector midpoint: the sweep tests depth and separation within the
balanced centered family, not lateral off-center placement.  Each
configuration was tested at
$\rho\in\{0,0.5,1,2,4,6,8,10\}$ by setting $\ell=\rho\sigma$ separately for
that configuration, for 104 comparisons with the full response.

Table~\ref{tab:sweep-errors} reports the aggregate errors over these comparisons.

\begin{table}[!htbp]
 \caption{Aggregate errors in the 104-pair parameter sweep.
 \label{tab:sweep-errors}}
 \centering
 \small
 \begin{tabular*}{0.94\textwidth}{@{\extracolsep\fill}lccc@{}}
  \toprule
  Quantity & Mean (\%) & 95th percentile (\%) & Maximum (\%)\\
  \midrule
  Full response vs. direct convolution of $u_m^a$ & 0.264 & 0.720 & 0.800\\
  Quadratic formula vs. full response & 3.280 & 7.364 & 10.389\\
  Cubic formula vs. full response & 0.234 & 0.672 & 1.697\\
  \bottomrule
 \end{tabular*}
\end{table}

Figure~\ref{fig:sweep-heatmap} shows all cases.  The quadratic error increases
with $\rho$ in every configuration and is largest for the shallow
$x_{c3}=10\,\mathrm{mm}$ target.  The cubic correction remains below $1\%$ in
102 of the 104 cases; its largest error, $1.697\%$, occurs for the shallow
target at $\rho=10$.  That configuration also has the largest
$\gamma_3=0.1502$ and the largest transport-profile error, so both asymptotic
layers are least favorable there.  Increasing depth reduces $\gamma_3$ and
reduces the mean and maximum errors of both approximations over the tested
lifetime ratios.  This aggregate trend is consistent with the local
large-distance analysis, but it is not a numerical proof of the conditional
orders in \eqref{eq:formal-error-orders}.  The absolute cubic error need not
decrease at every fixed ratio: at $\rho=6$, its values for depths 15, 20,
and 30~mm are $0.0161\%$, $0.1020\%$, and $0.1184\%$, respectively.
Across all 13 configurations, $0.0762\leq\gamma_3\leq0.1502$; hence every
cubic computation lies well inside the proven interval
$0\leq\gamma_3<1/2$ of Theorem~\ref{thm:cubic}.

The signed decomposition explains this nonmonotonicity.  Across the 104
pairs, the transport-profile and cubic local-shape errors have opposite
signs in 64 cases.  The mean absolute cubic local-shape difference,
normalized by the full peak time, is $0.217\%$, compared with the $0.264\%$
mean transport-profile error.  For the baseline geometry at $\rho=4$, the
signed components and their sum are
\[
 \underbrace{0.2391\%}_{\text{transport profile}}
 +\underbrace{(-0.5068\%)}_{\text{cubic local shape}}
 \simeq\underbrace{-0.2678\%}_{\text{total}}.
\]
Thus a small total error can partly reflect cancellation, and neither
component is an automatic lower bound for the total error.

\begin{figure}[!htbp]
 \centering
 \includegraphics[width=0.98\textwidth]{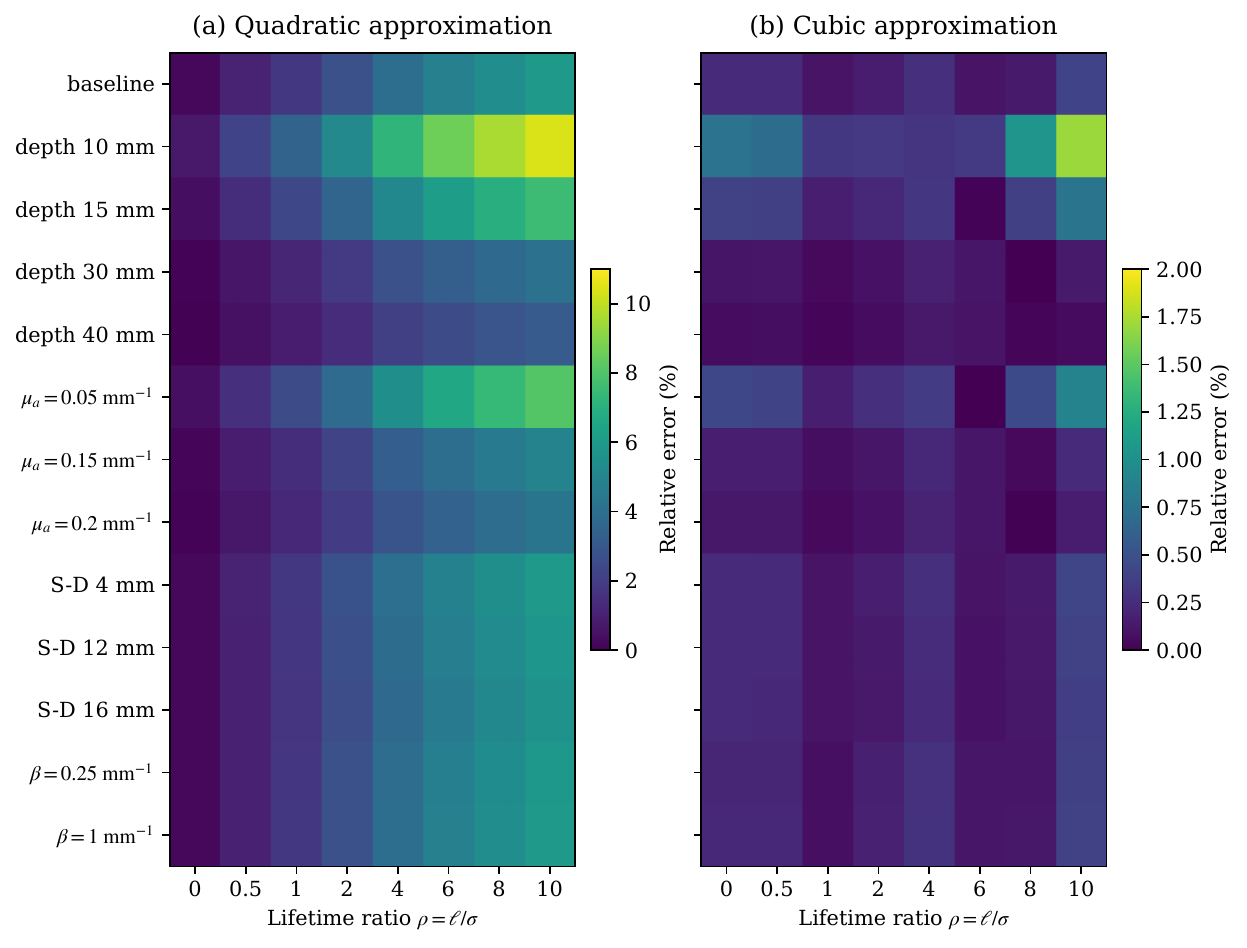}
 \caption{Relative errors for 13 one-at-a-time optical and centered-geometry
 configurations. Each row varies one quantity from the baseline. Panels (a)
 and (b) show the quadratic and cubic errors, respectively. Separate color
 scales are used so that the sub-percent structure of the cubic results
 remains visible}
 \label{fig:sweep-heatmap}
\end{figure}

\clearpage
\section{Discussion}
\label{sec:discussion}

The main physical conclusion is that fluorescence lifetime and photon
transport do not contribute two unrelated delays.  The zero-lifetime response
first establishes a center $t_0$, a width $\sigma$, and a skewness
$\gamma_3$.  The lifetime then acts relative to that width.  This explains why
the same numerical value of $\ell$ can be short in one geometry and long in
another, and why lifetime must be compared with the transport width when
moving between the endpoint regimes.

The quadratic result has two strengths: it is explicit, and it exposes this
dimensionless organization.  Its increasing long-lifetime bias is also
understandable.  An exponential lifetime kernel moves the measured peak into
the right side of the transport profile, where the locally symmetric Gaussian
becomes inaccurate.  The cubic correction measures precisely that first
asymmetry.  Across the 104-case study, the mean cubic error relative to the
full response, $0.234\%$, is comparable to the $0.264\%$ mean error caused by
replacing the full transport response with $u_m^a$.  However, the signed
diagnostic in Section~\ref{sec:numerics} shows that these components often
partly cancel.  The similarity of their aggregate errors therefore does not
establish that transport-profile error dominates the remaining discrepancy.
A quartic local term may help, particularly in the shallowest,
longest-lifetime cases, but its effect on total error must be assessed
together with the signed transport-profile contribution.  Systematic
improvement should address both approximation layers rather than infer a
preferred correction from the total mean error alone.

Several limitations should be kept visible.

\begin{enumerate}[(i)]
 \item The full numerical reference still uses a homogeneous half-space
 diffusion model and a point target.  Heterogeneous backgrounds, finite
 targets, curved boundaries, and radiative-transport effects are not included.
 \item The large-distance approximation $u_m^a$ is one source of error, while
 the quadratic or cubic local truncation is another.  The numerical
 decomposition shows that both worsen for shallow targets.
 \item The instrument response function, finite source pulse, detector timing
 offset, temporal binning, and photon-counting noise are absent.  In an
 experiment, a common timing offset should be calibrated or treated as a
 nuisance parameter, and the peak should be extracted by a local fit rather
 than by a raw maximum bin.
 \item Proposition~\ref{prop:peak-stability} gives a rigorous perturbation
 mechanism, but a global parameter-explicit error theorem for the complete
 response remains open.  Such a theorem would require simultaneous control of
 the large-distance remainder, the lifetime-shifted peak window, and root
 nondegeneracy.
\end{enumerate}

These limitations suggest a concrete validation sequence: first compare with
finite-element or Monte Carlo forward simulations, then convolve with measured
instrument-response functions, and quantify peak-picking uncertainty under
Poisson noise. The current formulas provide a fast and interpretable model
for carrying out those studies.

\section{Conclusion}
\label{sec:conclusion}

A single dimensionless ratio, $\rho=\ell/\sigma$, organizes the effect of
fluorescence lifetime on the peak of a time-domain diffuse-light response
from a fluorescent point target.  The center $t_0$ and width $\sigma$ are determined by the
zero-lifetime transport profile, while a universal mode displacement describes
the lifetime delay.  The resulting quadratic formula
\eqref{eq:quadratic-formula} defines one reduced prediction across short,
transitional, and long lifetimes without introducing an external lifetime
cutoff.  Its endpoint expansions match the leading structures of the
previous approximations in their joint large-distance regimes.  This
matching does not extend to the fixed-geometry limit
\eqref{eq:fixed-geometry-long-lifetime}.

The leading limitation of the quadratic formula is local symmetry.  A cubic
correction based on $\gamma_3$ retains the lifetime-unified structure and
requires only a scalar root connected to the quadratic mode.  It also
reproduces the short-lifetime expansion of the exact reduced-profile
convolution through order $\ell^2$.  It reduces the mean error from
$3.280\%$ to $0.234\%$ in a 104-pair one-at-a-time optical and
centered-geometry study, with the largest remaining error occurring in the
shallowest, longest-lifetime configuration.  On average, the remaining cubic
error is comparable to the error of the underlying transport-profile
approximation, but signed errors partly cancel in many cases.  The error
decomposition and conditional stability estimate therefore motivate
assessing both the local expansion and $u_m^a$ when pursuing further
systematic improvement.

Further work should establish parameter-explicit error control and assess
the effects of instrumental response and photon-counting noise. Applications
to inverse imaging require a separate analysis of identifiability and
reconstruction performance.

\section*{Statements and Declarations}

\subsection*{Funding}
This work was supported by the Japan Society for the Promotion of Science
(JSPS), KAKENHI Grant Number JP25K17291.

\subsection*{Competing interests}
The author declares no competing interests.

\subsection*{Author contributions}
Junyong Eom: Conceptualization; Methodology; Formal analysis; Software;
Validation; Visualization; Writing---original draft; Writing---review and
editing.

\subsection*{Data and code availability}
The Python code and synthetic numerical data supporting this study are supplied as ancillary files with this preprint. The script \path{anc/reproduce_results.py} generates the numerical results, and \path{anc/render_figures_from_csv.py} renders the figures from the included data. Package requirements and usage instructions are provided in \path{anc/README.txt}. No experimental or personal data were used.

\subsection*{Use of generative AI}
During the preparation of this manuscript, the author used ChatGPT (OpenAI) to assist with drafting and revising parts of the text, developing and checking Python code, and checking selected analytical derivations. The author reviewed and revised all AI-assisted material, independently reproduced the numerical results, and takes full responsibility for the final manuscript.

\subsection*{Ethics approval}
This study uses mathematical analysis and synthetic numerical data only;
ethics approval is not applicable.

\section*{Ancillary files}
Python reproduction code, synthetic numerical data, computational environment information, and usage instructions are included in the \path{anc/} directory accompanying this preprint.

\bibliographystyle{unsrtnat}
\bibliography{uniform_peak_time_references}

\end{document}